\documentclass[12pt]{amsart}
\usepackage{geometry}                
\usepackage[parfill]{parskip}    
\usepackage{graphicx}
\usepackage{amssymb}
\usepackage{epstopdf}
\newtheorem{thm}{Theorem}[section]

\newtheorem{lem}[thm]{Lemma}

\newtheorem*{thmBrown*}{Brown's criterion}
\newtheorem*{A*}{Theorem A}
\newtheorem*{B*}{Theorem B}

\newtheorem*{corollary*}{Corollary}

\newtheorem*{fpmconj*}{The $FP_m$-Conjecture}

\newtheorem{defn}[thm]{Definition}

\providecommand{\Ker}{\mathop{\rm Ker}\nolimits}
\providecommand{\Hom}{\mathop{\rm Hom}\nolimits}
\providecommand{\mod}{\mathop{\rm mod}\nolimits}

\title[Homological finiteness conditions]
{Homological finiteness conditions for a class of metabelian groups\footnote{\textrm{This work is supported by the EPSRC Grant N007238/1 on Soluble Groups and Cohomology. The first author acknowledges the support of the Institut Mittag-Leffler (Djursholm, Sweden). The second author acknowledges the support of EPSRC and the University of Glasgow. He subsequently held an Honorary Research Associateship at King's College London during the completion of this project.}}
}
\author{P. H. Kropholler and J. P. Mullaney}

\address{
   Peter H Kropholler\\
   Mathematical Sciences\\
   University of Southampton\\
   Southampton SO17 1BJ\\
   United Kingdom\\}
   \email{p.h.kropholler@soton.ac.uk\\}%
\address{
  Joseph P Mullaney\\
King's College London\\
Graduate School\\
150 Stamford Street\\
London SE1 9NH\\
United Kingdom}
\email{joseph.mullaney@kcl.ac.uk}

\subjclass[2010]{20J06}

\begin{document}
\maketitle

\begin{abstract}
We generalize a theorem of Groves and Kochloukova concerning cohomological finiteness conditions for metabelian groups in order to encompass a classical example of Baumslag and Stammbach. More precisely we shall show that for any natural number $n$ and indeterminate $x$, the group of $2\times 2$ matrices generated by 
\[\left (\begin{array} {cc}
1 & 0\\
1& 1\\
\end{array} \right ),\qquad	
\left(\begin{array}{cc}
x& 0\\
0&1\\
\end{array}\right),\qquad
\left (\begin{array} {cc}
n! & 0\\
0& 1\\
\end{array} \right ),\qquad
\left (\begin{array} {cc}
i+x & 0\\
0& 1\\
\end{array} \right )\]
for $i=1,\dots,n$ is of type $FP_{n+1}$ but not of type $FP_{n+2}$, thus  providing evidence in favour of the Bieri--Groves $FP_{m}$-conjecture. These examples are amongst the simplest known examples of torsion-free metabelian groups of their kind.
\end{abstract}

Work of Kochloukova and Groves \cite{GK06} establishes the Bieri--Groves $FP_m$ conjecture in some important special cases. We were motivated by wishing to extend these results to some examples of Baumslag, Kropholler, and Stammbach \cite{Ba72(2),BaS87,KrS90}. These examples are described in Section 2 below. The Kochloukova--Groves work does not quite fit the new context and the purpose of this paper is to show that this issue can be addressed by making some adjustments to the Kochloukova--Groves strategy. As a result, this paper follows theirs very closely, but in view of the elegance of the examples we believe it is important that this research be recorded.

\section{The Background and some Historical Remarks}

It was the work of Baumslag and Remeslennikov in the 1970s which suggested that the theory of finitely presented soluble groups might be richer than once thought. For example, it was demonstrated in \cite{Ba72(2)} that there exists a finitely presented metabelian group with a free abelian derived group of infinite rank. Remeslennikov \cite{R73} found the same example independently. This could be compared with the fact, proved in \cite{Ba73}, that every finitely generated metabelian group can be embedded in a finitely presented metabelian group. Recall that a group $G$ is said to be metabelian if there exists a short exact sequence
\begin{equation}
\label{metabelian}
A\hookrightarrow G \twoheadrightarrow Q
\end{equation}
of groups with $A$ and $Q$ abelian. Hence metabelian groups are precisely the soluble groups of derived length two.  

The group $G$ is finitely generated if and only if $Q$ is a finitely generated group and $A$ is finitely generated as a $\mathbb{Z}Q$-module via conjugation, and we will assume from now on that $G$ is finitely generated.
\par
In 1980 Bieri and Strebel \cite{BS80} characterized finitely presented metabelian groups in terms of the invariant  $\Sigma_A$. This is defined in the following way. A \emph{character} of $Q$ is a non-zero group  homomorphism $v: Q\rightarrow \mathbb{R}$; two characters are equivalent if they are positive real multiples of each other. The set  $S(Q)$ of all equivalence classes $[v]$ of characters of $Q$ can be identified with the unit sphere $S^{n-1}\subset {\mathbb{R}}^n\cong \Hom(Q, \mathbb{R})$, where $n$ is the $\mathbb{Z}$-rank of $Q$. Let $Q_{v}=\{q\in Q: v(q)\geq 0\}$. This is a monoid, so we can form the monoid ring $\mathbb{Z}Q_{v}$. We then associate to every finitely generated $\mathbb{Z}Q$-module $A$ the set
\[ \Sigma_A=\{[v] \in S(Q): A \textrm{ is finitely generated as a $\mathbb{Z}Q_{v}$-module }\}.\]

The $\mathbb{Z}Q$-module $A$ is then said to be $m$-\emph{tame} if every $m$-point subset of $\Sigma^c_A=S(Q)\setminus \Sigma_A$ lies in an open hemisphere of $S(Q)$. Equivalently, $v_1+\ldots+v_m\neq 0$ for any $[v_1],\ldots, [v_m]\in \Sigma^c_A$.

A group $G$ is said to be of type $FP_m$ if there is a $\mathbb{Z}G$-projective resolution of the trivial module $\mathbb{Z}$ where the modules are finitely generated in dimensions $\leq m$. In particular a group is of type $FP_1$ if and only if it is itself finitely generated. 
Bieri and Strebel \cite {BS80} proved that the properties of being of type $FP_2$ and being finitely presented are equivalent for metabelian groups. It is not known whether this is true for all soluble groups, but it is certainly not true for groups in general: for example, the right-angled Artin groups in \cite {BB97} are of type $FP_2$ but not finitely presented. In addition, Bieri and Strebel showed that $G$ is finitely presented if and only if the corresponding $\mathbb{Z}Q$-module $A$ is 2-tame. These results led to the $FP_m$-conjecture of Bieri and Groves.

\begin{fpmconj*}[\cite{BG82}]
\label{FP_m conjecture}
Suppose we have the short exact sequence (\ref{metabelian}) and let $m$ be a positive integer. Then $G$ is of type $FP_m$ if and only if $A$ is $m$-tame as a $\mathbb{Z}Q$-module.
\end{fpmconj*}

Both directions of the $FP_m$ conjecture remain open for $m>2$ but a number of specific cases have been proved. In \cite{BG82} it was shown that if $G$ is of type $FP_m$ then $A{\otimes}_{\mathbb{Z}} K$ is $m$-tame as a $KQ$-module for every field $K$. In
\cite{A86} Hans \AA berg established the full $FP_m$-conjecture for the case where $G$ has finite Pr\"{u}fer rank. Kochloukova \cite{K96} extended \AA berg's methods to show that the `only-if' part of the conjecture holds true if either the additive group of $A$ is torsion or if $G$ is the split extension of $A$ by $Q$. 

In general the status of the conjecture for non-split extensions seems to be more delicate and more technical than for split extensions, although as we shall see below, in many examples, there are few non-split extensions to be found.
In \cite {K96} it is shown that the full conjecture is true for the case where $A$ is torsion and of Krull dimension 1 as a $\mathbb {Z}Q$-module. More recently  Bieri and Harlander \cite{BH01} proved the $FP_3$-conjecture for the case where $G$ is the split extension of $A$ by $Q$.

\section{The examples of Baumslag--Stammbach and the desired generalization of the Groves--Kochloukova Theorem}

Baumslag's example \cite{Ba72(2)} of a finitely presented metabelian group with a free abelian normal subgroup of infinite rank was the group $G_1$ generated by the
$2\times 2$-matrices 

\[a=\left (\begin{array} {cc}
1 & 0\\
1& 1\\
\end{array} \right ),\qquad	
 s=\left (\begin{array}{cc}
1+x&0\\
0&1\\
\end{array} \right),	\qquad	
t=\left(\begin{array}{cc}
x& 0\\
0&1\\
\end{array}\right)\] 

over $\mathbb{Q}(x)$. This is a 3-generator 3-relator group with presentation
\[ G_1=\langle a,s,t: [s,t]=[a^t,a]=1, a^s=aa^t\rangle.\]
Fix a positive integer $n>1$. To this group we add the generators 

\[\left (\begin{array} {cc}
n! & 0\\
0& 1\\
\end{array} \right ),\qquad
\left (\begin{array} {cc}
i+x & 0\\
0& 1\\
\end{array} \right )\]

for all $2\leq i\leq n$. Using methods similar to Baumslag's it can be shown that the resulting metabelian group $G_n$ has a finite presentation and that its derived subgroup is torsion-free abelian of infinite rank.  We know that $G_n$ is isomorphic to the split extension $A_n\rtimes Q_n$ (and in fact by Theorem B below, every extension of $A_{n}$ by $Q_{n}$ is split), where 
\[A_n=\mathbb{Z}\left [x, x^{-1}, (1+x)^{-1},\ldots, (n+x)^{-1}, \frac{1}{n!}\right]\]
and $Q_n$ is the free abelian group with basis $\{q_{-1}, q_0,  q_1,\ldots,  q_n\}$ and an action $\circ$ on $A_n$ given by
\begin{equation} 
a\circ q_{-1}=(n!) a,\quad
a\circ q_0= xa,\quad
a\circ q_1=(1+x)a\quad
\dots \quad
a\circ q_n=(n+x)a
\end{equation}

for all $a\in A$. This action turns $A_n$ into a cyclic $\mathbb{Z}Q_n$-module. Now $A_n$ and $Q_n$ embed as subgroups of $G_n$ and so we can think of this action as conjugation in $G_n$ by the fact that
\[(1, a\circ q)= (q,0)^{-1}(1,a)(q,0).\] The $\mathbb{Z}Q_n$-module $A_n$ was first studied by Baumslag and Stammbach \cite{BaS87}, who showed that the exterior power ${\bigwedge}_{\mathbb{Z}}^i A_n\cong H_i(A_n,\mathbb{Z})$ is a non-free finitely generated  $\mathbb{Z}Q_n$-module for $1\leq i\leq n$, is a free $\mathbb{Z}[1/n!]Q_n$-module of finite rank for $i=n+1$ and is free of infinite rank for $i\geq n+2$. By studying the action of $G_n$ on a subcomplex of a product of trees indexed by a finite set of discrete (i.e. image $\mathbb{Z}$) characters of $G_n$ we can show that $G_n$ is a group of type $FP_{n+1}$. In fact we prove a more general result from which this can be deduced as a corollary.

\begin{A*} 
\label{Theorem A}
Let $Q=Q_0\times Q_1\times\ldots \times Q_l$ be a finitely generated free abelian group, where $Q_0=\langle q_{-1}\rangle$ and $Q_i$ is a free abelian group with basis ${\{q_{i,j}\}}_{0\leq j\leq z_i}$ for $1\leq i\leq l$. Let $A$ be a finitely generated (right) $\mathbb{Z}Q$-module and assume that the action of $\mathbb{Z}Q$ on $A$ factors through an action of a quotient $M=M_0\otimes M_1\otimes\ldots\otimes M_l$, where $M_i=\mathbb{Z}Q_i/I_i$, $I_0=\langle q_{-1}-k\rangle$ where $k$ is some positive integer and, for $1\leq i\leq l$, $I_i$ is generated as an ideal by ${\{q_{i,j}-f_{i,j}\}}_{0\leq j\leq z_i}$, where for fixed $i$ the $f_{i,j}$ are irreducible non-constant monic polynomials in $\mathbb{Z}[q_{i,0}]$ that are pairwise coprime in $\mathbb{Z}[q_{i,0}, 1/k]$ and $f_{i,0}=q_{i,0}$. Assume further that $A$ is free as an $M$-module. Then the split extension $G$ of $A$ by $Q$ is of type $FP_m$, where $m=\min\{rk(Q_i): 1\leq i\leq l\}$.
\end{A*}

\begin{corollary*}
\label{corollary to Theorem A}
The group $G_n$ is of type $FP_{n+1}$.
\end{corollary*}

\begin{proof} Note that no two of the polynomials $x, 1+x,\ldots, n+x$ generate a proper ideal in $\mathbb{Z}[x, 1/n!]$. Take $l=1$, $z_1=n$, $q_{1,j}=q_j$,  $f_{1,j}=x+j$ and $k=n!$. Then $M=M_0\otimes M_1\cong A_n$.
\end{proof}

The method of our proof originated with \AA berg \cite{A86} and closely follows that of a similar result by Groves and Kochloukova \cite[Theorem 5]{GK06}, of which ours is a generalization. In essence we have taken their result and showed it to be true when we invert an appropriate integer $k$ as well as inverting a set of polynomials that are pairwise coprime (in the sense that no two of them lie in a proper ideal of $\mathbb{Z}[q_{i,0},1/k]$.) The set of characters that we choose in our proof has the property that the set of equivalence classes of the restrictions of each character to $Q$ is contained in $\Sigma^c_A$. After constructing a $G$-tree for each character, and taking the product $X$ of these, we prove three properties of an appropriate subspace $Y$ of $X$:

(i) $G$ acts co-compactly on $Y$;\\
(ii) $Y$ is $(m-1)$-connected;\\
(iii) the stabilizers in $G$ of cells in $Y$ are of type $FP_m$.\\

That $G$ is of type $FP_m$ is then implied by a criterion of Brown \cite{KSB}. Hence the group $G_n$ is of type $FP_{n+1}$. It follows from Kochloukova's `only if' result \cite{K96} that $A_n$ is $(n+1)$-tame as a $\mathbb{Z}Q_n$-module, but by our choice of characters we see that it is not $(n+2)$-tame and so $G_n$ is not of type $FP_{n+2}$. In this way we have further evidence that the $FP_m$-conjecture may be true. 

It is desirable to have a direct proof of the tameness of $A_n$ and this will be the subject of a separate paper.

One consequence of the Bieri-Strebel theorems is that whether or not a metabelian group $G$ is finitely presented depends only upon the $Q$-module $A$, and not on the extension class in $H^2(Q,A)$. In particular, $G$ is finitely presented if and only if the split extension $A\rtimes Q$ is finitely presented. Using the Lyndon-Hochschild-Serre spectral sequence we have been able to show that $H^2(Q_n,A_n)=0$ and so \emph{every} extension of $A_n$ by $Q_n$ is split. In general we have the following.

\begin{B*}
\label{Theorem C}
Let $A$, $Q$ and the positive integer $k$ be as in Theorem A with $l=1$, and let $z_1=n$, $q_{1,j}=q_j$ and $f_{1,j}=f_j$. Then $H^2(Q,A)$ is cyclic of order dividing $k-1$. In the special case of $A_{n}$ and $Q_{n}$, the cohomology group vanishes. 
\end{B*} 

\section{The set $V$ of characters and the space $Y$}
From now on we are in the situation described in Theorem A. We shall write $q_j$ for $q_{1,j}$, $f_j$ for $f_{1,j}$ and $z_1=n$ so that $Q_1$ has rank $n+1$. Hence $n+1\geq m$. Let $\widetilde{M}=M_0\otimes M_1$; then
\[\widetilde{M}\cong\mathbb{Z}\left [q_0, {q_0}^{-1}, f_1^{-1},\ldots, f_n^{-1}, k^{-1}\right ]\]
where $q_j$ acts as multiplication by $f_j$ for $j\ge1$ and $q_{-1}$ acts as multiplication by the positive integer $k$. Let $\widetilde{Q}=Q_0\times Q_1$ and let $\phi:\mathbb{Z}\widetilde{Q}\rightarrow \widetilde{M}$ be the ring homomorphism sending $q_{-1}$ to $k$ and $q_i$ to $f_i$ for $0\leq i\leq n$. The restriction of $\phi$ to $\widetilde{Q}$ is a group homomorphism $\tau$ from the free abelian group $\widetilde{Q}$ to the group of units $(\widetilde{M})^{\times}$ of $\widetilde{M}$. Crucially, $\tau$ is injective,
so we may identify $\widetilde{Q}$ with its isomorphic copy in $(\widetilde{M})^{\times}$.  We shall think of $\widetilde{M}$ as a subring of the field of rational functions $\mathbb{Q}(q_0)$.

\begin{defn} Let $R$ be a non-trivial commutative ring and let $\mathbb{R}_{\infty}$ denote the set of real numbers together with an additional point $\infty$. Then a \emph{valuation} on $R$ is a map $v: R\rightarrow \mathbb{R}_{\infty}$ satisfying $v(0)=\infty$, $v(1)=0$, $v(ab)= v(a)+v(b)$ and $v(a+b)\geq \min\{v(a),v(b)\}$ for all $a,b \in R$.
\end{defn}

We then define $v_i$, for each $0\leq i\leq n$, to be the $f_i$-adic valuation on $\mathbb{Q}(q_0)$; that is, the unique valuation on $\mathbb{Q}(q_0)$ that is zero on $\mathbb{Q}\setminus \{0\}$ and satisfies $v_i(f_i)=1$ . In addition we define, for polynomials $g,h\in \mathbb{Z}[q_0]$,
\[v_*(g/h)= \deg(h)-\deg(g)\] and it is easily seen that this too is a valuation. So $v_*$ and the $v_i$ are group homomorphisms into the group of rational integers $\mathbb{Z}$ and their restrictions to $\widetilde{Q}$ provide us with a set $V$ of $n+2$ discrete characters of $\widetilde{Q}$. If in addition we define $v(Q_i)=0$ when $i>1$ for each $v\in V$ then we get a set of discrete characters of the group $Q$. Note that $v(Q_0)=0$. Each $v\in V$ can be extended to a character of $G$ via composition with the natural projection $\pi: G\twoheadrightarrow Q$.

Using the methods of Groves and Kochloukova in \cite{GK06} one constructs a tree $\Gamma_v$ corresponding to each $v\in V$. This tree is equipped with a map to $\mathbb R$ which is related to the valuation $v:Q\to\mathbb R$.  Thus there is a map $h:X\to\mathbb R^{n+1}$ where $X$ is the product of these trees and $n=rk(Q_1)-1$. 
In their reasoning, Groves and Kochloukova introduce an integer $\beta$, see \cite[\S4.1, Paragraph 3]{GK06}. In our context we need to
 set
\[ \beta= -\left(2+\sum_{i=0}^n d_i\right)\] where $d_i=\deg(f_i)$. 
We must also choose a set of generators $a_1,\ldots, a_d$ of $A$ as a $Q$-module that is a basis of $A$ as a free $M$-module. With this we have built $n+2$ trees $\Gamma_{v_*},\Gamma_{v_0},\ldots, \Gamma_{v_n}$ and from these we can construct the \AA berg CW-complex $Y=h^{-1}(W)$ described in \cite[\S4.2]{GK06}, where $W$ is the subspace spanned by those elements of $\mathbb R^{n+2}$ that arise as valuations of elements of $Q_1$.

By showing that $Y$ has the three properties
\indent
\begin{quote}
(i) {\bf [Cocompactness]} $G$ acts co-compactly on $Y$;
\newline
(ii) {\bf [Connectivity]} $Y$ is $(m-1)$-connected, that is $Y$ is path connected and the homotopy groups 
$\pi_1(Y),\ldots, \pi_{m-1}(Y)$
 are trivial; and
\newline
(iii) {\bf [Homological Finiteness]} the stabilizers in $G$ of cells in $Y$ are of type $FP_m$
\end{quote}
then Groves and Kochloukova can use the classical criterion of Brown.

\begin{thmBrown*}[(\cite{KSB} Proposition 1.1)]  Let $G$ be a group acting on a CW-complex $Y$ via permutation of the cells and such that for every cell the stabilizer of the cell fixes the vertices pointwise. Assume further that $Y$ is $(m-1)$-acyclic, the stabilizers in $G$ of cells of dimension $i\leq m$ are always of type $FP_{m-i}$ and $G$ acts co-compactly on $Y$. Then $G$ is of homological type $FP_m$.
\end{thmBrown*}

We reason in exactly the same way in this paper. 

\section{Cocompactness}

Observe that $W$ is the $(n+1)$-dimensional subspace 
\[W=\{(y_{-1},y_0,y_1,\ldots, y_n) \in \mathbb{R}^{n+2}:  y_{-1}+\sum_{i=0}^n d_i y_i=0\}.\]

\begin{defn}
For a subset $B$ of $\prod_{v\in V}\mathbb{R}$ we define
\[[[B]]=\left\{\prod_{v\in V} \lceil b_v\rceil: \prod_{v\in V} b_v \in B\right\}\]
where $\lceil b_v\rceil$ denotes the least integer greater than or equal to $b_v$.

\end{defn}

\begin{lem}[(\cite{GK06} Lemma 7)]
\label{Q acts co-finitely}
If  $\prod_{v\in V}s_v\in [[W]]$, then 
\[0\leq s_{v_*}+\sum_{i=0}^n d_i s_{v_i} <1+\sum_{i=0}^n d_i.\]
Hence $Q$ acts co-finitely on $[[W]]$.
\end{lem}

For $v\in V$ let $(Q_1)_v$ be the submonoid of $Q_1$ of all elements $q$ with $v(q)\geq 0$ and let $(M_1)_v$ be the $\mathbb{Z}(Q_1)_v$-submodule of $M_1$ generated by the image of $1_{Q_1}$ in $M_1$.

\begin{lem}[(\cite{GK06} Lemma 6)]
\label{monoid is positive half-sphere}
For $v\in V$ we have $(M_1)_v=\{m\in M_1: v(m)\geq 0\}$.
\end{lem}

\begin{thm} [(c.f. \cite{GK06} Theorem 6)] 
\label{adjustment}
For every $\prod_{v\in V}s_v$ in $[[W]]$ and every set 
\newline 
$\{a_v\in A: v\in V\}$, there exists an element $a\in A$ such that $\prod_{v\in V}[(a_v,s_v)]=\prod_{v\in V}[(a,s_v)]$. Consequently for every $\prod_{v\in V}r_v\in W$ such that $[[\prod_{v\in V}r_v]]=\prod_{v\in V}s_v$ one has $\prod_{v\in V}[(a_v,r_v)]=\prod_{v\in V}[(a,r_v)]$.

\end{thm}

\begin{proof} 
By the same argument as Theorem 6 in \cite{GK06}, we find that, using Lemma \ref{monoid is positive half-sphere}, it will suffice to find an $a\in A$ such that $v(a-a_v)\geq 0$ for each $v$. 
Let $F(q_0)=\prod_{i=0}^n f_i(q_0).$ Then, for some $t$, we have $F^t a_v\in \mathbb{Z}[q_0]$ for every $v\in\{v_0, v_1,\ldots, v_n\}$. Since the $f_i$ are all pairwise co-prime as elements of $\mathbb{Z}[q_0, 1/k]$ we can apply the Chinese Remainder Theorem for commutative rings to get a unique solution mod $F^t \mathbb{Z}[q_0, 1/k]$ to the congruences
\[a'\equiv a_{v_i}F^t \mod {f_i}^t  \mathbb{Z}[q_0, 1/k]\]
for $0\leq i\leq n$. This $a'\in A$ must have degree less than that of $F^t$. Set $a=a'F^{-t}$. Then
\[v_i(a-a_{v_i})=v_i(a'F^{-t}-a_{v_i})=v_i(F^{-t}(a'-a_{v_i} F^{t})=-t+v_i(a'-a_{v_i} F^{t}) \]
and since ${f_i}^t$ divides $a'-a_{v_i} F^t$ in $\mathbb{Z}[q_0, 1/k]$ we have $v_i(a-a_{v_i})\geq 0$. Recall that $a_w=0$. Thus
\[v_*(a-a_w)=v_*(a)=v_*(a'F^{-t})=\deg(F^t)-\deg(a')>0.\]

This completes the proof.

\end{proof}

\begin{thm} The group $G$ acts co-compactly on $Y$.
\end{thm}

\begin{proof}
This follows from Theorem \ref{adjustment}, using an argument identical to that of Theorem 7 in \cite{GK06}.
\end{proof}

\section{Connectivity}

\begin{lem}[(\cite{GK06} Lemma 8)] 
Every set of $m$ elements of $V$ lies in an open half-space of $\Hom(Q,\mathbb R)$
\end{lem}

Therefore, using \cite[Proposition III.3.3]{A86}, we have:

\begin{lem} $Y$ is $(m-1)$-connected.
\end{lem}

\section{Homological Finiteness}
It remains to show that if $P$ is the stabilizer in $G$ of a cell in $Y$ then $P$ is of type $FP_m$. The following lemma implies that it is enough to prove that the stabilizer of a vertex in $X$ lying in $h^{-1}[[W]]$ is of type $FP_m$. 

\begin{lem}[(\cite{K96} Lemma 2.9)] If $\Gamma$ is a cell of $Y$ then the stabilizer of $\Gamma$ in $G$ coincides with the stabilizer in $G$ of a vertex of $X$ lying in $h^{-1}([[W]])$. 
\end{lem}

By replicating the argument of Section 4.5 in \cite{GK06} we can see that an element $aq\in P$ if and only if
\[a\in A_{v_*}\circ {q_{v_*}}^{s_{v_*}+\beta}\cap A_{v_0}\circ {q_{v_0}}^{\beta}\cap \ldots \cap A_{v_n}\circ {q_{v_n}}^{\beta} \textrm{ and } q\in G(v) \textrm{ for all } v\in V.\]

The condition on $q\in Q$ is satisfied precisely when $q\in \Ker v$ for all $v$, i.e. whenever $q\in Q_0\times Q_2\times\ldots \times Q_l$. Since $A$ is free as an $M$-module it suffices to consider the case when $A$ is cyclic as an $M$-module, and so we can take $A=M$. Then $b\in A_v$ for all $v\neq v_*$ exactly if $b\in  \mathbb{Z}[q_0, 1/k]\otimes M_2\ldots\otimes M_l$. Let $\tilde{q}=q_0 q_1\ldots q_n\in Q_1$. Then since $a\in {\bigcap}_{0\leq i\leq n} A_{v_i}\circ {q_{v_i}}^{\beta}$ we have $a\circ (\tilde{q})^{-\beta}\in A_{v_i}$ for $0\leq i\leq n$ and so $a\circ (\tilde{q})^{-\beta}\in  \mathbb{Z}[q_0, 1/k]\otimes M_2\ldots\otimes M_l$. Write the component of $a\circ (\tilde{q})^{-\beta}$ in $\mathbb{Z}[q_0, 1/k]$ as $g(q_0)$. Now 
\[a\in A_{v_*}\circ {q_0}^{-(s_{v_*}+\beta)}\subseteq\{b\in A: {v_*}(b)\geq s_{v_*}+\beta\}.\] It follows that $s_{v_*}+\beta\sum_{i=0}^n d_i\leq {v_*}(a\circ (\tilde{q})^{-\beta})<0$, and so the degree of $g(q_0)$ is bounded above by $-(s_{v_*}+\beta\sum_{i=0}^n d_i)=d$. Hence the component $g(q_0)\circ (\tilde{q})^{\beta}$ of $a$ that is a polynomial in $\mathbb{Z}[q_0, 1/k]$ is a $\mathbb{Z}$-linear combination of the elements  
\[(k)^{l_j}{q_0}^{j+\beta}{f_1}^{\beta}\ldots{f_n}^{\beta}\]
where $0\leq j \leq d$ and the $l_j$ are the powers of $k$ appearing in the monomials of $g$. Thus $P\cap\widetilde{M}$ is a free $\mathbb{Z}[1/k]$-module of finite rank $d+1$ and so $P\cap A=P\cap (\mathbb{Z}[q_0, 1/k]\otimes M_2\otimes\ldots\otimes M_l)$ is a finitely generated free $M'=\mathbb{Z}[1/k]\otimes M_2\otimes\ldots\otimes M_l$-module. Hence, $P$ is the split extension of a free $\mathbb{Z}[1/k]\otimes M_2\otimes\ldots\otimes M_l$-module by $Q_0\times Q_2\times\ldots\times Q_l$.

We now perform induction on $l$ to show that $P$ is always of type $FP_m$ and so $G$ is of type $FP_m$, where $m=\min\{rk(Q_i): 1\leq i\leq l\}$.

\begin{thm} Every stabilizer $P$ in $G$ of a cell in the \AA berg complex $Y$ is of type $FP_m$, and so $G$ is of type $FP_m$.
\end{thm} 
\begin{proof} First suppose $l=1$. Then $P$ is the split extension of a free $\mathbb{Z}[1/k]$-module of finite rank $d+1$ by the infinite cyclic group $Q_0$. Since the group $\mathbb{Z}[1/k]\rtimes Q_0$ has the presentation $\langle x,t: t^{-1}xt=x^k\rangle$ we deduce that $P=(P\cap A)\rtimes Q_0$ has the presentation
\[\langle H, t: t^{-1}x_i t={x_i}^{k} \textrm{ for }0\leq i\leq d\  \rangle\]
where $H=\langle x_0,\ldots,x_d\rangle$ is free abelian group of rank $d+1$. Hence $P$ is an HNN-extension of $H$ with stable letter $t$, and the base group $H$ and associated subgroups are free abelian of finite rank. It follows from \cite[Proposition 2.13(b)]{B76} that $P$ is of type $FP_\infty$ and so is certainly of type $FP_m$. Hence by Brown's criterion $G$ is of type $FP_m$.

We now assume that Theorem A holds when $M'$ is a tensor product of $\mathbb{Z}[1/k]$ and $l-1$ other components. In particular then $P$ is of type $FP_{m'}$, where $m'=\min\{rk(Q_i): 2\leq i\leq l\}$, and so $P$ is of type $FP_m$, since $m\leq m'$. This completes the proof of Theorem A.
\end{proof}

Hence we have proved that $G_n$ is of type $FP_{n+1}$. To see that this statement is sharp, we appeal to \cite[Theorem B]{K96} where it is proved that the `only if' direction of the Bieri--Groves conjecture holds in the split extension case. This implies that $A$ is $(n+1)$-tame. However, using the fact that the $v_i$ (as characters of $Q$) are restrictions of valuations, we deduce from \cite[Theorem 2.1]{BS81} that it is not $(n+2)$-tame, since \[{v_*}+v_0+d_1v_1+\ldots+d_nv_n=0\] and $[v_*], [v_0],\ldots, [v_n]\in\Sigma_A^c$ , and $A\neq A_v$ for all $v\in V$.

\section{Group extensions}
Suppose we are in the $l=1$ case of Theorem A, and write $q_j= q_{1,j}, f_j=f_{1,j}$ and $z_1=n$. Suppose also that $A$ is a cyclic $M$-module; then $A\cong \mathbb{Z}[q_0, q_0^{-1}, f_1^{-1}, \ldots, f_n^{-1}, k^{-1}]$. We shall use the Lyndon-Hochschild-Serre spectral sequence to calculate the second cohomology group $H^2(Q,A)$. First of all note that we have a short exact sequence
\[\langle q_{-1}\rangle=B\longrightarrow Q\longrightarrow C=\langle q_0, q_1,\ldots, q_n\rangle,\] and recall that $q_{-1}$ is acting as multiplication by the integer $k$. In the calculation below we assume that $k\ge2$.

It is clear that $H^0(B,A)=0$.
Since $B$ is infinite cyclic $H^{n}(B,A)=0$ for $n\geq2$, and $H^1(B,A)=H_{0}(B,A)=(\mathbb{Z}/(k-1)\mathbb{Z})[q_0, q_0^{-1}, f_1^{-1},\ldots, f_n^{-1}]$. In order to calculate $H^2(Q,A)$ we must first calculate $H^1(C, H^1(B,A))$. We get another LHS-spectral sequence via the short exact sequence
\[\langle q_0\rangle=X\longrightarrow C\longrightarrow C'=\langle q_1,\ldots, q_n\rangle\] and use the fact that $q_{0}$ acts as multiplication by $q_{0}$.

Now $H^0(X, H^1(B,A))=0$ since only the zero element in $H^1(B,A)$ is fixed under the action of $X$. Moreover we have 
\[ H^1(X, H^1(B,A))=H_{0}(X, H^1(B,A))=(\mathbb{Z}/(k-1)\mathbb{Z})[ f_1(1)^{-1},\ldots, f_n(1)^{-1}].\] The group $H^1(C, H^1(B,A))$, and thus $H^2(Q,A)$, is then given by the fixed points in $(\mathbb{Z}/(k-1)\mathbb{Z})[ f_1(1)^{-1},\ldots, f_n(1)^{-1}]$ under the action of $C'$, that is $H^{0}(C',\mathbb{Z}/(k-1)\mathbb{Z})$. In general it is cyclic of order dividing $k-1$. 
In the application to extensions of $A_{n}$ by $Q_{n}$, we have that $k=n!$ and each $q_{i}\  (i>0)$ may be identified with an integer (specifically, the integer $f_i(1)$) in such a way that $q_i-1$ divides $n!$ (and hence $q_i-1$ becomes a unit modulo $k-1$). Therefore in this case the cohomology group vanishes and every extension of $A_{n}$ by $Q_{n}$ is split.

Thus we have proved Theorem B: what this shows is that the delicacy of the $FP_{m}$ conjecture for non-split extensions is further complicated by the difficulty of finding simple examples where non-split extensions exist and can be used to test the theory.

\subsection{Concluding Remark}
We thank the referee for a very careful reading of this paper and for suggesting many improvements.

\end{document}